\documentclass[11pt]{article}
\usepackage{amsmath, amssymb, amsfonts, amsthm, graphicx}
\theoremstyle{plain}
\newtheorem{theorem}{Theorem}[section]
\newtheorem{lemma}[theorem]{Lemma}

\theoremstyle{definition}
\newtheorem{definition}[theorem]{Definition}
\newtheorem{corollary}[theorem]{Corollary}

\newtheorem{example}[theorem]{Example}
\begin{document}
\title{Periods of the discretized Arnold's Cat Mapping and its extension to $n$-dimensions}
\author{Joe Nance\\
  \texttt{nance2uiuc@gmail.com}\\
  Department of Mathematics \\
  University of Illinois at Urbana-Champaign\\}
\maketitle

\begin{abstract}
A discrete dynamical system known as Arnold's Discrete Cat Map (Arnold's DCM) is given by
\[
(x_{t+1},y_{t+1})=(x_t+y_t,x_t+2y_t)\mod N,
 \]
 which acts on a two-dimensional square coordinate grid of size $N\times N$. The defining characteristic of this map is that it has the property that when the $N\times N$ grid is a picture whose pixels are assigned $(x,y)$ coordinates, the map scrambles the picture with each iteration. After a finite number of iterations, the picture is restored to its original shape and order. The number of iterations needed to restore the image $M$, has a mysterious dependence on $N$. This period, as we will find out, is directly related to the divisibility of the Fibonacci numbers. We will exploit this property to show that, for any $N$, an image is not dense in itself. In the second half of the paper, we build on the work of Chen, Mao, and Chui to extend the DCM to three dimensions. Finally, we define the generalized $n$ dimensional DCM by introducing the idea of a ``matrix union".
\end{abstract}
\section{The nature of the 2D DCM}
Arnold's Cat Mapping is a chaotic mapping on the two dimensional torus given by
\[A(x,y)=(x+y,x+2y)\mod 1,\]
but for our purposes, we will use the discretized version \[(x_{t+1},y_{t+1})=(x_t+y_t,x_t+2y_t)\mod N,\] where $N$ is the size of the discrete grid on which the map acts. The 2D DCM can be written equivalently in matrix form as
\[
\begin{pmatrix}
x_{t+1} \\
y_{t+1}\end{pmatrix}
=
\begin{pmatrix}
1 & 1 \\
1 & 2\end{pmatrix}
\begin{pmatrix}
x_t \\
y_t\end{pmatrix}\mod N.
\]

\begin{example}
Consider an arbitrary $3\times 3$ image given by
\[\begin{pmatrix}
A & B & C \\
D & E & F \\
G & H & I \end{pmatrix}.\]
The $(x,y)$-coordinates of the pixels in this image are
\[\begin{pmatrix}
(0,2) & (1,2) & (2,2) \\
(0,1) & (1,1) & (2,1) \\
(0,0) & (1,0) & (2,0) \end{pmatrix} \mod 3.\]
Let's look at the orbit of an arbitrary pixel under the DCM,
\[(1,1) \rightarrow (2,0) \rightarrow (2,2) \rightarrow (1,0) \rightarrow (1,1).\] The orbit of the entire image looks like
\[\begin{split}\begin{pmatrix}
A & B & C \\
D & E & F \\
G & H & I \end{pmatrix}
\rightarrow
\begin{pmatrix}
B & D & I \\
F & H & A \\
G & C & E \end{pmatrix}
& \rightarrow
\begin{pmatrix}
D & F & E \\
A & C & B \\
G & I & H \end{pmatrix} \rightarrow \\
& \rightarrow
\begin{pmatrix}
F & A & H \\
B & I & D \\
G & E & C \end{pmatrix}
\rightarrow
\begin{pmatrix}
A & B & C \\
D & E & F \\
G & H & I \end{pmatrix}.\end{split}\]
So a $3\times 3$ image has period $M=4$ since  four iterations of the cat map were needed to return the image to its original state.
\end{example}
\begin{table}[h!]
\begin{center}
\begin{tabular}{| l| c| l |}
\hline
pixel dimension of image ($N \times N$) & iterations to restore image (period) \\
\hline
300 $\times$ 300 & 300 \\
257 $\times$ 257 & 258 \\
183 $\times$ 183 & 60 \\
157 $\times$ 157 & 157 \\
150 $\times$ 150 & 300 \\
147 $\times$ 147 & 56 \\
124 $\times$ 124 & 15 \\
100 $\times$ 100 & 150 \\
\hline
\end{tabular}
\end{center}
\caption{Notice, there is not an apparent correlation between the size of an image and its period.}
\end{table}
Dyson and Falk give some relationships and bounds on the period of an image of size $N$:
\[M \begin{cases}
=3N \mbox{ for } N=2(5^s) \ ; s\in \mathbb{N}\\
=2N \mbox{ for } N=5^s \mbox{ or } 6(5^s) \ ; s\in \mathbb{N} \\
\leq \frac{12N}{7}\text{ for other }N\end{cases}\]

\section{An image is not dense in itself}
As in Example 1.1, consider the orbit of a single pixel as it jumps around an image with each iteration. An interesting question to ask is, ``Does there exist $N$ such that every pixel visits every other pixel at least once?" or perhaps more formally, ``Does there exist $N$ such that an image is dense in itself?"

To answer this question, it will be convenient to use the matrix
\[A=\begin{pmatrix}
0 & 1 \\
1 & 1\end{pmatrix},
\mbox{ where }
A^2=:A_2=\begin{pmatrix}
1 & 1 \\
1 & 2\end{pmatrix},\]
and define the Fibonacci sequence as $u_0=0,\ u_1=1,\ u_2=1,\ u_3=2, u_4=3,...,[u_{t+2}=u_{t+1}+u_t]$. Then, the matrix for the $t$-th iteration of the DCM is given by
\[A^{2t}=\begin{pmatrix}
u_{2t-1} & u_{2t}\\
u_{2t} & u_{2t+1}\end{pmatrix}.\]

For a given $N$, the period $M$ of the DCM is the smallest positive integer $t$, such that $u_{2t}\equiv 0\mod N$, and $u_{2t-1}\equiv 1\mod N$. This implies that $u_{2t+1}\equiv u_{2t+2}\equiv 1\mod N$. Thus, the period $M$ of the DCM for a given image size $N$ is strongly related to the divisibility of the Fibonacci numbers. Dyson and Falk give the following theorem and proof in [2]. The fact that an image is not dense in itself follows as a corollary.

\begin{theorem}
For an image of size $N\times N$, where $N \geq 3$, the period $M$ of this image satisfies $M\leq N^2/2$.
\end{theorem}
Before we can prove this theorem, we will prove three short lemmas.

\begin{definition}
Define $\phi_i$ as the least non-negative residue of $u_i$, the $i$th Fibonacci number $\mod N$. That is, $u_i\equiv\phi_i\mod N$.
\end{definition}
Consider the sequence of ordered pairs $\langle\phi_1,\phi_2\rangle,\langle\phi_2,\phi_3\rangle,\dots,\langle\phi_i,\phi_{i+1}\rangle,\dots$ There are at most $N^2$ distinct pairs. Any set of $N^2+1$ pairs contains some equal ones among them.

\begin{lemma}
The first pair that repeats in the above sequence is $\langle1,1\rangle$.
\end{lemma}

\begin{proof}
Assume the opposite: that the first repeated pair is $\langle\phi_k,\phi_{k+1}\rangle$, where $k>1$. So let us find a pair $\langle\phi_r,\phi_{r+1}\rangle$ where $r>k$ in the sequence such that $\phi_k=\phi_r,\phi_{k+1}=\phi_{r+1}$. From the definition of the Fibonacci numbers,
\begin{align*}
\phi_{r-1} &=\phi_{r+1}-\phi_r, \\
\phi_{k-1} &=\phi_{k+1}-\phi_k,\end{align*}

so $\phi_{r-1}=\phi_{k-1}$ and we have that $\langle\phi_{r-1},\phi_r\rangle=\langle\phi_{k-1},\phi_k\rangle$.
But $\langle\phi_{k-1},\phi_k\rangle$ occurs earlier in the sequence than $\langle\phi_k,\phi_{k-1}\rangle$; therefore $\langle\phi_k,\phi_{k-1}\rangle$ is not the first pair that repeats itself. So the supposition $k>1$ is wrong. Therefore, $k=1$. This proves the lemma.
\end{proof}

\begin{lemma}
For any positive integer $N$, at least one number divisible by $N$ can be found among the first $N^2$ Fibonacci numbers.
\end{lemma}
\begin{proof}
From the previous lemma, $\langle 1,1\rangle$ is the first pair that repeats itself. So $\langle \phi_t, \phi_{t+1}\rangle=\langle 1,1 \rangle$ for some integer $t$ such that $1<t\leq N^2+1$. Thus
\[\phi_t \equiv 1\mod N,\]
and
\[\phi_{t+1} \equiv 1\mod N.\]
But
\[u_{t-1}=u_{t+1}-u_t.\]
Therefore,
\[\phi_{t-1} \equiv 0\mod N.\]
This proves the lemma.
\end{proof}
\begin{lemma}
For $N>2$, if $u_t \equiv 0\mod N$ and $u_{t+1} \equiv 1\mod N$, then $t$ must be even.
\end{lemma}
\begin{proof}
The lemma is equivalent to the statement that for $N>2$, if $A^t \equiv 1\mod N$, then $t$ is even. But $\det(A)=-1$, so $\det(A^t)=(\det\ A)^t=(-1)^t \equiv 1\mod N$. Hence $t$ must be even. This proves the lemma.
\end{proof}
\begin{theorem}
The period, $M$ of the DCM satisfies $M\leq \frac{N^2}{2}$.
\end{theorem}
\begin{proof}
From the first and second lemmas, the second occurrence of the pattern $0,1,1$ in the sequence $\phi_0,\phi_1,...,\phi_t,\phi_{t+1},...$ happens for $\phi_{j-1}, \phi_j, \phi_{j+1}$, where $0<j-1\leq N^2$. From the third lemma, $j-1$ must be even. From the definition of the period, we have that $2M=j-1$. This proves the theorem.
\end{proof}
\begin{corollary}
An image is not dense in itself.
\end{corollary}
\begin{proof}
Suppose there exists an $N$ such that an $N\times N$ image was dense in itself. For this to be true, it would have to be the case that the period of the image equals $N^2-1$ since there are $N^2$ pixels and without any iterations, the pixel occupies itself. But from the theorem, an arbitrary pixel in an $N\times N$ image has a period at most of half the total number of pixels in the image; a contradiction. Therefore, the number of unique places that an arbitrary pixel occupies in one full cycle is never more than half of the number of possible places. Since the choice of pixel is arbitrary, this theorem holds for the whole image. So an image is not dense in itself under the discrete cat map. This proves the corollary.
\end{proof}

\section{Higher dimensional analogs}
What would a three dimensional analog to the $2D$ DCM look like? Certainly, the mapping would have to act on an $N\times N\times N$ cube. An $N$-cube would have to posses a finite period, like its $2D$ cousin. The mapping would also have to posses the area-preserving and mixing dynamics of the $2D$ DCM. A plausible way to obtain an explicit form for the $3D$ DCM would be to compose three mappings. Each mapping would fix one coordinate of $x,y,z$-space and have the lower dimensional DCM act on the remaining two coordinates. This way, after composition of the three ``basis" maps, the entire $N\times N\times N$ cube has been mapped in the way of the cat.

We have
\[\mathbf{A_2}=\begin{pmatrix}
1 & 1 \\
1 & 2 \end{pmatrix}.\]
Fixing $x$, we obtain the first of three ``basis" maps,
\[a_3^1=\begin{pmatrix}
1&0&0 \\
0&1&1 \\
0&1&2 \end{pmatrix}.\]
Fixing $y$, we obtain the second basis map,
\[a_3^2=\begin{pmatrix}
1&0&1\\
0&1&0 \\
1&0&2 \end{pmatrix}.\]
Fixing $z$, the third basis map is
\[a_3^3=\begin{pmatrix}
1&1&0\\
1&2&0 \\
0&0&1 \end{pmatrix}.\]

Composition of linear functions dictates that we multiply their respective matrix representations,
\[\mathbf{A_3}=a_3^1 a_3^2 a_3^3=\begin{pmatrix}
1&1&1\\
2&3&2 \\
3&4&4 \end{pmatrix}.\]

This is the matrix for the three dimensional discrete cat map.

\subsection{The concept of ``matrix union"}
The systematic extrapolation of the $2D$ DCM to three dimensions can be generalized for an arbitrary positive integer dimension. We just need to introduce a bookkeeping device to formalize the action of fixing a coordinate in a basis map.
\begin{definition} $\{F_n^i\}$ is the set of ``$n\times n$ $i$-frames",
\[F_n^i=\begin{pmatrix}
 & & &0& & & \\
 & & &\vdots & & & \\
 & & &0& & & \\
0&\cdots &0& d_i &0&\cdots &0\\
 & & &0& & & \\
 & & &\vdots & & & \\
 & & &0& & & \end{pmatrix},\]
where $d_i=1$ is the $i$th diagonal and the $i$th row and column except for the entry $[d_i]$ consist entirely of $0$'s.
\end{definition}

\begin{example}
\[F_3^2=\begin{pmatrix}
 &0& \\
0&1&0 \\
 &0& \end{pmatrix},
 \
 F_3^1=\begin{pmatrix}
1&0&0\\
0& & \\
0& & \end{pmatrix}\]
\end{example}

\begin{definition}
The matrix union $\mho$ is a binary function $\mho: \mathbf{A_n} \times F_{n+1}^i \rightarrow a_{n+1}^i$ which takes an $n$ dimensional DCM matrix $\mathbf{A_n}$ and inserts it into an $n+1$ dimensional $i$-frame. The output is an $n+1$ dimensional basis map $a_{n+1}^i$, which fixes the $i$th coordinate of the $n$ dimensional DCM.
\end{definition}

\begin{definition}
A basis map $a_{n+1}^i$ is an element of the ``union basis" of $\mathbf{A_n}$, $\{a_{n+1}^i\}={\{\mho(\mathbf{A_n},F_{n+1}^i)\}}_{i=1}^{n+1}$. This matrix fixes the $i$th coordinate of the $n$ dimensional DCM.
\end{definition}

\begin{example}
\[\mho (\mathbf{A_2}, F_3^1)=a_3^1=\begin{pmatrix}
1&0&0 \\
0&1&1 \\
0&1&2\end{pmatrix},
\
\mho (\mathbf{A_3}, F_4^3)=a_4^3=\begin{pmatrix}
1&1&0&1\\
2&3&0&2\\
0&0&1&0\\
3&4&0&4\end{pmatrix}.\]
\end{example}
This new machinery allows for a concise iterative expression for the matrix for the $n$ dimensional DCM:
\begin{equation}\label{eq:ndcm}
\mathbf{A_{n+1}}=\prod_{i=1}^{n+1} \mho(\mathbf{A_n},F_{n+1}^i) ;\ n\geq 2.
\end{equation}
\begin{example}
\textrm{Generate the matrix for the $4D$ DCM.}
Using (\ref{eq:ndcm}), we have
\[a_4^1=\begin{pmatrix}
1&0&0&0 \\
0&1&1&1 \\
0&2&3&2 \\
0&3&4&4\end{pmatrix}
a_4^2=\begin{pmatrix}
1&0&1&1\\
0&1&0&0 \\
2&0&3&2 \\
3&0&4&4\end{pmatrix}
a_4^3=\begin{pmatrix}
1&1&0&1\\
2&3&0&2 \\
0&0&1&0 \\
3&4&0&4\end{pmatrix}
a_4^4=\begin{pmatrix}
1&1&1&0 \\
2&3&2&0 \\
3&4&4&0 \\
0&0&0&1\end{pmatrix}.\]
And so
\[\mathbf{A_4}=a_4^1 a_4^2 a_4^3 a_4^4=\begin{pmatrix}
17&23&18&5\\
110&149&117&31 \\
257&348&274&72 \\
432&585&460&122\end{pmatrix}.\]
\end{example}
\subsection{Dynamical properties of higher dimensional cat maps}
We now focus our attention on the general cat map whose $t$-th iteration is given by $\vec{x}_t=\mathbf{A_n}^t\vec{x}_0$, where $\mathbf{A_n}$ is the matrix derived in the preceding section. In this mapping, we allow $\vec{x}_0$ to be continuous. Unlike the DCM, this map allows for points with irrational coordinates. This way, the mapping has chaotic orbits.
\begin{definition}
(Alligood, Saur, and Yorke Def. 5.2)
Let $\mathbf{f}$ be a map of $\mathbb{R}^m$, $m\geq 1$, and let $\{\mathbf{v}_0, \mathbf{v}_1, \mathbf{v}_2, ...\}$ be a bounded orbit of $\mathbf{f}$. This orbit is chaotic if: it is not asymptotically periodic, $\mathbf{f}$ has no eigenvalue equal to $1$, and at least one eigenvalue is greater than $1$ in absolute value.
\end{definition}
For an orbit of the general cat map, an argument using folliations of the torus (which is beyond the scope of this paper) can be used to show that if a point has irrational coordinates, then its orbit is not asymptotically periodic. For the second two conditions, it boils down to finding the characteristic polynomial of $\mathbf{A_n}$. If $(x-1) \nmid \chi_{\mathbf{A_n}}(\lambda)$ and at least one root of $\chi_{\mathbf{A_n}}(\lambda)$ is greater than $1$ in absolute value, then any orbit of the general cat map whose coordinates are irrational is chaotic. Characteristic polynomials are given for the first few cases of the $n$ dimensional general cat map below:

\begin{align*}
\chi_{\mathbf{A_2}}(\lambda) &= \lambda^2-3\lambda+1, \\
\chi_{\mathbf{A_3}}(\lambda) &= -\lambda^3+8\lambda^2-6\lambda+1, \\
\chi_{\mathbf{A_4}}(\lambda) &= \lambda^4-562\lambda^3+410\lambda^2-66\lambda+1.\end{align*}

Approximate forms for eigenvalues are given respectively as:
\begin{align*}
&0.381966,\quad 2.61803; \\
&0.243019,\quad 0.572771,\quad 7.18421; \\
&0.0168808,\quad 0.209427,\quad 0.50397,\quad 561.27. \end{align*}

As you can see, so far there are not any eigenvalues equal to $1$ and at least one eigenvalue is larger than $1$ in absolute value in each case. These eigenvalues can be loosely interpreted as the amount of stretching on a per iteration basis of the nearby orbits under the map, where the largest value dominates the stretching behavior. In a sense, these numbers measure how chaotic a system is.

I conjecture that as the dimension of the generalized cat map increases, the dominant eigenvalue of the map increases. Hence, larger dimensional analogs are more chaotic than their lower dimensional counterparts. Since we have an iterative way using the matrix union to generate higher dimensional matrices for an arbitrary cat map which factors as a product of easily understood basis maps, I propose that in order to prove or disprove this, one would need to develop some way of studying the interaction of characteristic polynomials under the action of multiplying their corresponding matrices. This way, one may be able to establish a bijection between the natural numbers and absolute values of eigenvalues of cat maps of increasing dimension.


\newpage
\begin{thebibliography}{widest entry}
\bibitem{cite_key1} Peterson, Gabriel. ``Arnold's Cat Map." College of the Redwoods. Ed. David Mills. N.p., Sept. 1997. Web. 3 Mar. 2011.
\bibitem{cite_key2} F. Dyson and H. Falk. ``Period of a Discrete Cat Mapping." \emph{The American Mathematical Monthly} 99.7 Aug. (1992): 603-14.
\bibitem{cite_key3} Chen, Mao, and Chui. ``A symmetric image encryption scheme based on 3D chaotic cat maps." \emph{Chaos, Solitons, and Fractals}1.21 (2004): 749-61.
\bibitem{cite_key4} Alligood, Kathleen T., Tim D. Saur, and James A. Yorke. Chaos: An Introduction to Dynamical Systems. New York: Springer, 1996.
\bibitem{bar} C. Bar, Elementary Differential Geometry, New York: Cambridge University Press, 2010.
\bibitem{ymb} Y.M. Baryshnikov, Spherical billiards with periodic orbits, preprint.
\bibitem{birkhoff}  G.D. Birkhoff, Dynamical Systems, revised edition, Amer. Math. Soc. Colloq. Publ., vol. IX, Amer. Math. Soc., Providence, RI, 1966.
\bibitem{carmo} M. do Carmo, Differential Geometry of Curves and Surfaces, Prentice–Hall Inc., Englewood Cliffs, New Jersey (1976).
\bibitem{genin} D. Genin, S. Tabachnikov. On configuration space of plane polygons,
 sub-Riemannian geometry and periodic orbits of outer billiards.
 J. Modern Dynamics, 1 (2007), 155-173.
\bibitem{tumanov} A. Tumanov and V. Zharnitsky, Periodic orbits in outer billiard, International Mathematics Research Notices, vol. 2006, 1-17.
\end{thebibliography}
 \end{document}